\documentclass[11pt]{amsart}
\usepackage{amssymb,latexsym}

\def\triangle{\Delta}

\def\be1{{\begin{equation}}}
\def\ee1{{\end{equation}}}

\def\part{\partial}

\def\ba{\begin{array}}
\def\ea{\end{array}}

\numberwithin{equation}{section}

\newtheorem{lemma}{Lemma}[section]

\newtheorem{theorem}[lemma]{Theorem}

\title[A monotonicity formula and Type-II singularities ]
{A monotonicity formula and Type-II singularities for the mean curvature flow}

\author{Yongbing Zhang}

\address{School of Mathematical Sciences\\
         USTC\\
         Hefei, 230026, Anhui Province, China.}

\email{ybzhang@amss.ac.cn}

\begin{document}
\maketitle

\begin{abstract}
In this paper, we introduce a monotonicity formula for the mean curvature flow.
We also apply this monotonicity formula to study the asymptotic behavior of eternal solutions.
\end{abstract}

\section{Introduction}

The study of singularities plays an important role in understanding the global nature of mean curvature flow and finding geometric applications.
According to Huisken \cite{Huisken}, singularities for the mean curvature flow are divided into Type-I and Type-II.
A fundamental tool for the blowup analysis of Type-I singularities is Huisken's monotonicity formula \cite{Huisken},
and as a consequence one gets a self-shrinking mean curvature flow.
In this paper we introduce another monotonicity formula, which indicates some kind of connection between Type-II singularity and
translating solution to the mean curvature flow.

Let $F:\Sigma\times [0,T)\rightarrow \mathbb{R}^N$ be a smooth solution to the mean curvature flow
$\frac{d}{dt}F(p,t)=H$ and $M_t=F(\Sigma,t)$.
If $\Sigma$ is a closed manifold and $V$ any constant vector field on $\mathbb{R}^N$,
Ilmanen \cite{Ilmanen94} proved the following monotonicity formula
\begin{equation}\label{compactmonotonicity0}
\frac{d}{dt}\int_{M_t} e^{<x-tV,V>}d\mu_t=-\int_{M_t} |H-V^\perp|^2e^{<x-tV,V>}d\mu_t,
\end{equation}
where $x=F(p,t)$ and $V^\perp$ is the normal projection of $V$.

A surface in $\mathbb{R}^N$ is called a translating soliton if its mean curvature vector is equal to the normal part of some $V$, i.e. $H=V^\perp$.
However there exist no closed translating solitons in $\mathbb{R}^N$, and in general for noncompact $\Sigma$ the integral
$\int_{M_t} e^{<x-tV,V>}d\mu_t$ is not finite.
For this reason and the need for applications to study eternal solutions, we come up with a local formulation.
Let $\Sigma$ be a complete manifold and $M_t$ a solution to the mean curvature flow.
For any constant vector field $V$ on $\mathbb{R}^N$, we consider a reparametrization of the mean curvature flow defined by
\begin{equation}\label{reparamcf0}
\frac{d}{dt}F(p,t)=H+V^T,
\end{equation}
here $V^T$ is the tangential part of $V$.
Let $\widetilde{\Sigma}$ be any relatively compact domain of $\Sigma$ and $\widetilde{M}_t=F(\widetilde{\Sigma},t)$
where $F$ obeys the flow (\ref{reparamcf0}).
Our (local) monotonicity is
\begin{equation}\label{monotonicity}
\Phi_{V}(\widetilde{\Sigma},t):=\int_{\widetilde{M}_t}e^{<x-tV,V>}d\mu_t.
\end{equation}

\begin{theorem}
Under the evolution (\ref{reparamcf0}) we have the monotonicity formulas
\begin{equation}\label{measuremonotonicity}
\frac{d}{dt}(e^{<F(p,t)-tV,V>}d\mu_t)=-|H-V^\perp|^2e^{<F(p,t)-tV,V>}d\mu_t,
\end{equation}
\begin{equation}\label{localmonotonicity}
\frac{d}{dt}\Phi_{V}(\widetilde{\Sigma},t)=-\int_{\widetilde{M}_t} |H-V^\perp|^2e^{<x-tV,V>}d\mu_t.
\end{equation}
\end{theorem}

One sees that $\Phi_{V}(\widetilde{\Sigma},t)$ is non-increasing in $t$ along (\ref{reparamcf0}),
and it is steady for every $\widetilde{\Sigma}$ if and only if $M_t$ is a translating solution with $H=V^\perp$.
The translating solution has arisen from Type-II singularities of the mean-convex mean curvature flow,
see \cite{Hamilton95, HuiskenSinestrari1, HuiskenSinestrari2}.
However for a general Type-II singularity, the limiting flow of rescaled flows is an eternal solution only \cite{HuiskenSinestrari1},
see also Section 2.
Comparing to the correspondence between Type-I singularity and self-shrinking mean curvature flow,
a tight link between Type-II singularity and translating solution, in case that there really exists, had not been found.

The form of (\ref{localmonotonicity}) exhibits some kind of privilege of translating solutions,
similar to the role played by a self-shrinking solution
in Huisken's monotonicity formula \cite{Huisken}.
As mentioned in above, the limiting flow of the rescaled flows near a Type-II singularity is an eternal solution to the mean curvature flow.
The following indicates a direct connection between eternal solution and translating solution.

\begin{theorem}\label{eternal-translatingsolution}
Let $F:\Sigma \times \mathbb{R}\rightarrow \mathbb{R}^N$ be an eternal solution to (\ref{reparamcf0}), i.e. $\frac{d}{ds}F(p,s)=H+V^T$.
Assume the non-increasing 1-parameter family of measures $e^{<F(p,s)-sV,V>}d\mu_s$
on $\Sigma$ converges to finite and positive measures as $s\rightarrow\pm \infty$, i.e.
\begin{equation}\label{finitepositiveassumption0}
0<\lim_{s\rightarrow\infty}e^{<F(p,s)-sV,V>}d\mu_s\leq \lim_{s\rightarrow -\infty}e^{<F(p,s)-sV,V>}d\mu_s<\infty.
\end{equation}
Then there exist $b_k\rightarrow \infty$ and $a_k\rightarrow -\infty$ such that
$M_{b_k}$ and $M_{a_k}$ are asymptotic to translating solitons in the following sense:
for any relatively compact domain $\widetilde{\Sigma}$,
\begin{equation}\label{future0}
\int_{\widetilde{M}_{b_k}}|H-V^\perp|^2e^{<x-b_kV,V>}d\mu\rightarrow 0,
\end{equation}
\begin{equation}\label{ancient0}
\int_{\widetilde{M}_{a_k}}|H-V^\perp|^2e^{<x-a_kV,V>}d\mu\rightarrow 0.
\end{equation}
\end{theorem}

This is a straightforward application of (\ref{localmonotonicity}) to eternal solutions.
Note that for a translating solution with $H=V^\perp$, (\ref{reparamcf0}) gives rise to $F(p,s)=F(p,0)+sV$.
In particular along (\ref{reparamcf0}), $e^{<F(p,s)-sV,V>}$ and $d\mu_s$ are independent of $s$.
Theorem \ref{eternal-translatingsolution} provides a sufficient condition, i.e. (\ref{finitepositiveassumption0}),
which guarantees the existence of a translating soliton
with inherited properties from the original mean curvature flow.

Along Theorem \ref{eternal-translatingsolution}, we propose the following general questions:
(1) Can we find $V$ such that (\ref{finitepositiveassumption0}) is satisfied for some specific mean curvature flows?
(2) When is an eternal solution asymptotic to a translating soliton in the sense of (\ref{future0}) and (\ref{ancient0})
in fact a translating solution?

The remaining part of this paper is arranged as follows.
In Section 2, we provide with necessary notations and background.
In Section 3, we prove the monotonicity formulas.
In the last Section, we prove Theorem \ref{eternal-translatingsolution}.

We would like to thank Kaole Si for pointing out to us the reference \cite{Ilmanen94}.

\section{Notations and background}

In this section, we briefly present the notations.
We then give a rough summary of some fundamental concerning singularities
of the mean curvature flow and blowup analysis at singularities. See for instance the survey \cite{Smoczyk}.

Let $F:\Sigma^n\times [0,T)\rightarrow \mathbb{R}^N$ be a smooth maximal solution to the mean curvature flow
\begin{equation}\label{mcf}
\frac{d}{dt}F(p,t)=H(F(p,t),t), \quad p\in \Sigma.
\end{equation}
Let $(x^i)_{i=1}^n$ denote local coordinates on $\Sigma$, $F_i=dF_t(\frac{\partial}{\partial x^i})$
be a local frame of the tangent space of $M_t:=F_t(\Sigma)$.
The second fundamental form of $M_t$ is denoted by
$A_{ij}=(\overline{\nabla}_{F_i}F_j)^\perp$,
where $\overline{\nabla}$ is the Levi-Civita connection of the ambient space $\mathbb{R}^N$.
Let $\overline{g}$ or $<\cdot,\cdot>$ be the Euclidean metric on $\mathbb{R}^N$
and $g_{ij}:=<F_i,F_j>=F_t^*(\overline{g}|_{M_t})(\frac{\partial}{\partial x^i},\frac{\partial}{\partial x^j})$
denote the induced metric on $(\Sigma,t)$.
Then the mean curvature vector field is $H=g^{ij}A_{ij}$ and the norm of the second fundamental form is $|A|^2:=g^{ik}g^{jl}<A_{ij},A_{kl}>$.
Via the isometry $F_t^*$, $d\mu_t$ denotes volume elements on both $M_t$ and $(\Sigma,t)$.

If $\Sigma$ is an $n$-dimensional closed manifold, the first singular time $T<\infty$.
Moreover at time $T$, $\limsup_{t\rightarrow T} \max_{M_t}|A|^2= \infty$.
Huisken \cite{Huisken} distinguishes first time singularities by the blowup rate of $\max_{M_t}|A|^2$.
If there exists a positive constant $C$ such that for all $0\leq t<T$ and $p\in \Sigma$, it holds that
$$(T-t)|A(p,t)|^2\leq C,$$
one says that the mean curvature flow (\ref{mcf}) develops a Type-I singularity. Otherwise one says that
(\ref{mcf}) develops a Type-II singularity, where $\sup_{x\in M_t}|A(x,t)|^2(T-t)=\infty$.

To study a singularity one usually perform rescalings around the space-time singularity.
A point $x_0\in \mathbb{R}^N$ is called a blowup point, if there exists a sequence $p_k \in \Sigma$ such that
$$\lim_{t\rightarrow T}F(p_k,t) = x_0,\quad \lim_{t\rightarrow T}|A(p_k, t)| =\infty.$$
For any $\lambda>0$, one can define a parabolic dilation of $M_t$, which is centered at $(x_0,T)$, by
\begin{equation}\label{centralblowup}
M_{\lambda,s}=\lambda(M_{T+\lambda^{-2}s}-x_0), \quad s\in [-\lambda^2T,0).
\end{equation}
$M_{\lambda,s}$ is a solution to the non-parametrized mean curvature flow and is called a rescaled flow.
For a Type-I singularity $(x_0,T)$, there exists a sequence $\lambda_j\rightarrow \infty$ such that $M_{\lambda_j,s}$
converges smoothly to a limiting flow, denoted by $M_{\infty,s}$, which is a self-shrinking mean curvature flow.
Huisken's monotonicity formula \cite{Huisken} is the reason why the limiting flow is a self-shrinking mean curvature flow.
For a Type-II singularity $(x_0,T)$, Ilmanen \cite{Ilmanen} and White \cite{White} show that there exists a sequence of
rescaled flows $M_{\lambda_j,s}$, defined also by (\ref{centralblowup}), which converges weakly to a limiting flow.
However one of the disadvantages is that in general the limiting flow is not smooth any more.

Following Hamilton's idea in Ricci flow \cite{Hamilton93}, Huisken and Sinestrari \cite{HuiskenSinestrari1}
introduce a rescaling procedure for Type-II singularities descried below.
One can first choose an essential blowup sequence $(x_k, t_k)$, i.e.
for any $k\geq  1$, let $t_k \in [0, T - 1/k]$ and $x_k \in M_{t_k}$ be such that
$$|A(x_k,t_k)|^2(T-\frac{1}{k}-t_k)=\max_{t\leq T-\frac{1}{k},x\in M_t}|A(x,t)|^2(T-\frac{1}{k}-t).$$
Let
$$L_k=|A(x_k,t_k)|, \quad \alpha_k=-L_k^2t_k, \quad \omega_k=L_k^2(T-t_k-\frac{1}{k}).$$
Huisken and Sinestrari show that for a singularity of Type-II, the following holds
$$\omega_k\rightarrow +\infty,\quad L_k\rightarrow +\infty, \quad t_k\rightarrow T,  \quad \alpha_k\rightarrow -\infty.$$
One can then consider the following rescaled mean curvature flows
\begin{equation}\label{HSrescaling}
M_{k,s}:=L_k(M_{t_k+L_k^{-2}s}-x_k), \quad s\in [\alpha_k,\omega_k].
\end{equation}
For each rescaled flow (\ref{HSrescaling}), $0\in M_{k,0}$ and $|A(0,0)|=1$.
Moreover, there exists a subsequence of rescaled flows which converges smoothly on every compact set to a limiting flow $M_{\infty,s}$.
$M_{\infty,s}$ is a solution to the mean curvature flow defined for all $s\in \mathbb{R}$, i.e. an eternal solution.
Moreover the eternal solution $M_{\infty,s}$ has the properties that $0\in M_{\infty,0}$, $|A|\leq 1$ and $|A(0,0)|=1$.

Type-II singularities of the mean curvature flow starting from a closed hypersurface in $\mathbb{R}^{n+1}$ with positive scalar mean curvature
has been well understood through the works \cite{Hamilton95,HuiskenSinestrari2}.
In this setting, Huisken and Sinestrari \cite{HuiskenSinestrari2} show that the limiting flow
arising from a Type-II singularity is a convex eternal solution.
By employing a Harnack inequality, Hamilton \cite{Hamilton95} had shown that
any strictly convex eternal solution to the mean curvature flow where the mean curvature assumes its maximum value at a point in space-time
must be a translating soliton.

\section{Monotonicity formulas}

In this section, we prove the monotonicity formulas.
Let $F:\Sigma\times [0,T)\rightarrow \mathbb{R}^N$ be a solution to the non-parametrized mean curvature flow $(\frac{dF}{dt})^\perp=H$
and $x=F(p,t)$ denote the position. In the sequel, $V$ always stands for a constant vector field on $\mathbb{R}^N$.
We first consider the case that $\Sigma$ is a closed manifold. Set
$$\Phi_{V}(t)=\int_{M_t} e^{<x-tV,V>}d\mu_t=\int_\Sigma e^{<F(p,t)-tV,V>}d\mu_t.$$
The functional $\int_Me^{<x,V>}d\mu$ had been introduced by Ilmanen in \cite{Ilmanen94}. 
Shahriyari \cite{Shahriyari} utilizes the functional in the study of graphic translating solitons in $\mathbb{R}^3$.

\begin{theorem}[Ilmanen \cite{Ilmanen94}]
Let $\Sigma$ be a closed manifold and $M_t=F_t(\Sigma)$ a solution to the flow $(\frac{d}{dt}F)^\perp=H$.
Then for any $V$, 
\begin{equation}\label{compactmonotonicity}
\frac{d}{dt}\Phi_{V}(t)=-\int_{M_t} |H-V^\perp|^2e^{<x-tV,V>}d\mu_t.
\end{equation}
\end{theorem}

\begin{proof}
Note that $\Phi_{V}(t)$ is independent of the parametrization of $M_t$.
One can assume for simplicity that $F(p,t)$ is a solution to the mean curvature flow $\frac{d}{dt}F=H$.
Hence
\begin{eqnarray*}
\frac{d}{dt}\Phi_{V}(t)&=&\int_{M_t} (<H,V>-|V|^2-|H|^2)e^{<x-tV,V>}d\mu_t.
\end{eqnarray*}
On the other hand, we have
\begin{eqnarray*}
0&=&\int_{M_t}  \triangle (e^{<x-tV,V>})d\mu_t
\\&=&\int_{M_t} (<H,V>+|V^\top|^2)e^{<x-tV,V>}d\mu_t
\\&=&\int_{M_t} (<H,V>+|V|^2-|V^\perp|^2)e^{<x-tV,V>}d\mu_t.
\end{eqnarray*}
By adding the above two identities, we get
\begin{eqnarray*}
\frac{d}{dt}\Phi_{V}(t)&=&\int_{M_t} (2<H,V>-|V^\perp|^2-|H|^2)e^{<x-tV,V>}d\mu_t
\\&=&-\int_{M_t} |H-V^\perp|^2e^{<x-tV,V>}d\mu_t.
\end{eqnarray*}
\end{proof}

Note that for a closed $\Sigma$, the monotonicity (\ref{compactmonotonicity}) is strict.
This is due to the simple fact that any complete translating soliton $M$ with $H=V^\perp$ must be non-compact.
In fact by $H=V^\perp$,
$$\int_M <H,V>d\mu=\int_M |H|^2d\mu >0,$$
which contradicts with, if $M$ is closed,
$$\int_M <H,V>d\mu=\int_Mdiv(V^T)d\mu=0.$$

In order to study the limiting eternal solution $M_{\infty,s}$ of rescaled flows (\ref{HSrescaling}),
what we need is a local version of (\ref{compactmonotonicity}).
Let $M_t$ be a mean curvature flow with $\Sigma$ complete,
we consider a reparametrization of the mean curvature flow and define
\begin{equation}\label{reparamcf}
\frac{d}{dt}F(p,t)=H+V^T.
\end{equation}
For any relatively compact domain $\widetilde{\Sigma}$ of $\Sigma$,
let $\widetilde{M}_t=F(\widetilde{\Sigma},t)$ where $F$ satisfies the flow (\ref{reparamcf}).
We define
\begin{equation}\label{localmonoquantity}
\Phi_{V}(F_t(\widetilde{\Sigma}),t)=\int_{\widetilde{\Sigma}}e^{<F(p,t)-tV,V>}d\mu_t.
\end{equation}
Or in another equivalent form,
\begin{equation*}
\Phi_{V}(F_t(\widetilde{\Sigma}),t)=\int_{\widetilde{M_t}} e^{<x-tV,V>}d\mu_t.
\end{equation*}
We simply use $\Phi_{V}(\widetilde{\Sigma},t)$ to denote (\ref{localmonoquantity}).
For any given $t$, $\Phi_{V}(\widetilde{\Sigma},t)$ is independent of the internal parametrization of $\widetilde{M}_t$.
However when $t$ varies, $\Phi_{V}(\widetilde{\Sigma},t)$
depends on the reparametrization of the non-parametrized mean curvature flow $(\frac{dF}{dt})^\perp=H$.
To put an emphasis on the reparametrization given by the tangential part of $V$, we call (\ref{reparamcf})
a mean curvature V-flow, or simply a V-flow.

We now describe a scaling property of the monotonicity (\ref{localmonoquantity}).
Let $M_t$ be a solution to the non-parametrized mean curvature flow $(\frac{dx}{dt})^\perp=H$.
Consider a reparametrization $F(p,t)$ of $M_t$ and a rescaled flow defined by
$$F_\lambda(p,s)=\lambda[F(p,t_0+\lambda^{-2}s)-x_0],$$
where $\lambda>0$,  $x_0\in \mathbb{R}^N$ and $t_0<T$.
Let $t=t_0+\lambda^{-2}s$. Then $F_\lambda(p,s)$ is a solution to the mean curvature V-flow
$$\frac{d}{ds}F_\lambda(p,s)=H(F_\lambda(p,s))+V^T$$
if and only if
$$\frac{d}{dt}F(p,t)=H(F(p,t))+(\lambda V)^T.$$
Note that
\begin{eqnarray*}
\int_{\widetilde{\Sigma}}e^{<F_\lambda(p,s)-sV,V>}d\mu_s
&=&\int_{\widetilde{\Sigma}}e^{<\lambda(F(p,t)-x_0)-\lambda^2(t-t_0)V,V>}\lambda^nd\mu_t
\\&=&\lambda^n\int_{\widetilde{\Sigma}}e^{<(F(p,t)-x_0)-(t-t_0)\lambda V,\lambda V>}d\mu_t,
\end{eqnarray*}
hence for the monotonicity (\ref{localmonoquantity}) one has the following scaling property
\begin{equation}\label{scalingproperty}
\Phi_V(F_{\lambda,s}(\widetilde{\Sigma}),s)=\lambda^n\Phi_{\lambda V}(F_t(\widetilde{\Sigma})-x_0,t-t_0).
\end{equation}

\begin{theorem}
Let $\widetilde{\Sigma}$ be any relatively compact domain of $\Sigma$ and $V$ a constant vector field on $\mathbb{R}^N$,
then under the V-flow (\ref{reparamcf}) we have
\begin{equation}\label{pointwisemonotonicity1}
\frac{d}{dt}(e^{<F(p,t)-tV,V>}d\mu_t)=-|H-V^\perp|^2e^{<F(p,t)-tV,V>}d\mu_t,
\end{equation}
\begin{equation}\label{localmonotonicity1}
\frac{d}{dt}\Phi_{V}(\widetilde{\Sigma},t)=-\int_{\widetilde{\Sigma}} |H-V^\perp|^2e^{<F(p,t)-tV,V>}d\mu_t=-\int_{\widetilde{M}_t} |H-V^\perp|^2e^{<x-tV,V>}d\mu_t.
\end{equation}
\end{theorem}
\begin{proof}
Under the evolution (\ref{reparamcf}), the induced metric on $\widetilde{\Sigma}$ evolves according to
$$\frac{d}{dt}g_{ij}(p,t)=-2<H-V,A_{ij}>.$$
Hence we have
$$\frac{d}{dt}d\mu_t=-<H-V,H>d\mu_t$$
and
$$\frac{d}{dt}(e^{<F(p,t)-tV,V>}d\mu_t)=-|H-V^\perp|^2e^{<F(p,t)-tV,V>}d\mu_t.$$
Then
\begin{eqnarray*}
\frac{d}{dt}\Phi_{V}(\widetilde{\Sigma},t)
&=&\frac{d}{dt}\int_{\widetilde{\Sigma}}e^{<F(p,t)-tV,V>}d\mu_t
=-\int_{\widetilde{\Sigma}} |H-V^\perp|^2e^{<F(p,t)-tV,V>}d\mu_t.
\end{eqnarray*}
\end{proof}

\section{Eternal solution and translating soliton}

We now apply the monotonicity formulas (\ref{pointwisemonotonicity1}) and (\ref{localmonotonicity1})
to study eternal solutions and prove Theorem \ref{eternal-translatingsolution}.
Let $M_s$ be an eternal solution of the mean curvature flow and reparametrize it so that it satisfies the mean curvature V-flow
$$\frac{d}{ds}F(p,s)=H+V^T.$$

\begin{theorem}
Let $F:\Sigma \times \mathbb{R}\rightarrow \mathbb{R}^N$ be an eternal solution to the V-flow.
Assume the non-increasing 1-parameter family of measures $e^{<F(p,s)-sV,V>}d\mu_s$ on $\Sigma$ converges to finite and positive measures as
$s\rightarrow\pm \infty$, i.e.
\begin{equation}\label{finitepositiveassumption}
0<\lim_{s\rightarrow\infty}e^{<F(p,s)-sV,V>}d\mu_s\leq \lim_{s\rightarrow -\infty}e^{<F(p,s)-sV,V>}d\mu_s<\infty.
\end{equation}
Then there exist $b_k\rightarrow \infty$ and $a_k\rightarrow -\infty$ such that
$M_{b_k}$ and $M_{a_k}$ are asymptotic to translating solitons in the following sense:
for any relatively compact domain $\widetilde{\Sigma}$,
\begin{equation}\label{future}
\int_{\widetilde{M}_{b_k}}|H-V^\perp|^2e^{<x-b_kV,V>}d\mu\rightarrow 0,
\end{equation}
\begin{equation}\label{ancient}
\int_{\widetilde{M}_{a_k}}|H-V^\perp|^2e^{<x-a_kV,V>}d\mu\rightarrow 0.
\end{equation}
\end{theorem}

\begin{proof}
By (\ref{pointwisemonotonicity1}), $e^{<F(p,s)-sV,V>}d\mu_s$ is non-increasing.
Recall that along a V-flow we have the monotonicity formula (\ref{localmonotonicity1}), i.e.
\begin{eqnarray*}
\frac{d}{ds}\int_{\widetilde{\Sigma}}e^{<F(p,s)-sV,V>}d\mu_s&=&
-\int_{\widetilde{\Sigma}}|H-V^\perp|^2e^{<F(p,s)-sV,V>}d\mu_s.
\end{eqnarray*}
It implies that for any $s_1<s_2$,
\begin{eqnarray}\label{s1s2}
&&\int_{s_1}^{s_2}\int_{\widetilde{\Sigma}}|H-V^\perp|^2e^{<F(p,s)-sV,V>}d\mu_sds\nonumber
\\&=&\int_{\widetilde{\Sigma}}e^{<F(p,s)-sV,V>}d\mu_s|_{s=s_1}-\int_{\widetilde{\Sigma}}e^{<F(p,s)-sV,V>}d\mu_s|_{s=s_2}.
\end{eqnarray}
Hence for any given $s_1$, we have
$$\int_{s_1}^{+\infty}\int_{\widetilde{\Sigma}}|H-V^\perp|^2e^{<F(p,s)-sV,V>}d\mu_sds<\infty.$$
Then for any given $\epsilon>0$ there exists some $b$ which can be chosen arbitrarily large such that
\begin{equation*}
\int_{\widetilde{M}_{b}}|H-V^\perp|^2e^{<x-bV,V>}d\mu_b=\int_{\widetilde{\Sigma}}|H-V^\perp|^2e^{<F(p,b)-bV,V>}d\mu_{b}<\epsilon.
\end{equation*}
This proves (\ref{future}).

By the assumption (\ref{finitepositiveassumption}),
$\lim_{s\rightarrow -\infty}\int_{\widetilde{\Sigma}}e^{<F(p,s)-sV,V>}d\mu_s$ exists and is finite.
Then (\ref{ancient}) follows in a similar way from (\ref{s1s2}).
\end{proof}

\end{document}